\documentclass[reqno]{amsart}
\usepackage{amssymb}
\usepackage[usenames, dvipsnames]{color}
\usepackage{hyperref}

\usepackage{verbatim}

\newtheorem{theorem}{Theorem}[section]

\newtheorem{lemma}[theorem]{Lemma}
\newtheorem{proposition}[theorem]{Proposition}

\theoremstyle{definition}
\newtheorem{remark}[theorem]{Remark}

\theoremstyle{definition}
\newtheorem{definition}[theorem]{Definition}

\theoremstyle{definition}
\newtheorem{assumption}[theorem]{Assumption}

\makeatletter
\def\dashint{\operatorname%
{\,\,\text{\bf--}\kern-.98em\DOTSI\intop\ilimits@\!\!}}
\makeatother

\def\bR{\mathbb{R}}
\def\bN{\mathbb{N}}

\def\cB{\mathcal{B}}

\newcommand{\Div}{\operatorname{div}}

\title[Quasilinear conormal derivative problem]{Global regularity of solutions to quasilinear conormal derivative problem with controlled growth}

\author[D. Kim]{Doyoon Kim}
\address[D. Kim]{Department of Applied Mathematics, Kyung Hee University,
1 Seocheon-dong, Giheung-gu, Yongin-si, Gyeonggi-do 446-701 Republic of Korea}
\email{doyoonkim@khu.ac.kr}
\thanks{This work was supported by a grant from the Kyung Hee University in 2010. (KHU-20101825)}

\subjclass[2010]{35J62, 35J25, 35B65}

\keywords{quasilinear elliptic equations, conormal derivative boundary value problem, BMO coefficients, Sobolev spaces.}

\begin{document}
	
\begin{abstract}
We prove the global regularity of weak solutions to the conormal derivative boundary value problem for quasilinear elliptic equations in divergence form on Lipschitz domains under the controlled growth conditions on the low order terms.
The leading coefficients are in the class of BMO functions with small mean oscillations.
\end{abstract}

\maketitle

\section{Introduction}

We consider the conormal derivative boundary value problem 
\begin{equation}							 \label{eq3.24}
\left\{
  \begin{aligned}
    - D_i\left(A_{ij}(x,u) D_j u + a_i(x,u)\right) &= b(x,u,\nabla u)
 \quad \text{in} \quad \Omega,\\
    \left(A_{ij}(x,u) D_j u + a_i(x,u)\right) \cdot \nu(x) &= 0 \quad \text{on} \quad \partial \Omega.
  \end{aligned}
\right.
\end{equation}
Here the equation is a quasilinear elliptic equation in divergence form, $\Omega$ is a bounded Lipschitz domain in $\bR^d$, $d \ge 2$, with a small Lipschitz constant,
and $\nu(x)$ is the outward normal vector to the surface $\partial \Omega$.
We call $u\in W^1_2(\Omega)$ a weak solution to \eqref{eq3.24} if
$$
\int_\Omega (A_{ij}(x,u)D_ju+a_i(x,u)) D_i\phi\,dx=\int_\Omega b(x,u,\nabla u)\phi\,dx
$$
for any $\phi\in W^1_2(\Omega)$.

In this paper we study the global regularity of weak solutions to \eqref{eq3.24} under the {\em controlled growth conditions} on $a_i$ and $b$.
First of all, the nonlinear terms $A_{ij}(x,u)$, $a_i(x,u)$, $b(x,u,\xi)$
in \eqref{eq3.24} are of Carathe\'{o}dory type, i.e., they are measurable in $x \in \bR^d$ for all $(u,\xi) \in \bR^d$,
and continuous in $(u,\xi) \in \bR \times \bR^d$ for almost all $x \in \bR^d$.
The leading coefficients $A_{ij}$ are bounded and uniformly elliptic, that is, for some constant $\mu\in (0,1]$,
\begin{equation}
								\label{eq110409_02}
|A_{ij}| \le \mu^{-1},
\quad
A_{ij}\xi_i \xi_j \ge \mu |\xi|^2 \quad \forall\xi\in \bR^d.	
\end{equation}
We also assume that $A_{ij}(x,u)$ are uniformly continuous in $u$ and have small mean oscillations with respect to $x$.
It is well-known that functions in this class are not necessarily continuous.
Throughout the paper, we set
\begin{equation}
								\label{eq110411_04}
\gamma =
\left\{
\begin{aligned}
&\frac{2d}{d-2}, \quad d > 2,\\
&\text{any number bigger than $2$}, \quad d = 2.
\end{aligned}
\right.	
\end{equation}
By the controlled growth conditions, we mean 
$$
|a_i(x,u)| \le \mu_1 (|u|^{\lambda_1} + f),
\quad
|b(x,u,\nabla u)| \le \mu_2 (|\nabla u|^{\lambda_2} + |u|^{\lambda_3} + g)
$$
for some constants $\mu_1,\mu_2>0$, where
$\lambda_1 = \gamma/2$, $\lambda_2 = 2(1-1/\gamma)$,
$\lambda_3 = \gamma-1$, and
$$
f\in L_2(\Omega),\quad g\in L^{\frac \gamma {\gamma-1}}(\Omega).
$$
Since $u \in W_2^1(\Omega)$ implies $u \in L_\gamma(\Omega)$, the controlled growth conditions guarantee the convergence of the integrals in the definition of weak solutions above.
If $1 \le \lambda_1 < \gamma/2$, $1 \le \lambda_2 < 2(1-1/\gamma)$,
$1 \le \lambda_3 < \gamma-1$,
we say that the equation \eqref{eq3.24} satisfies the {\em strictly controlled growth conditions}.
As mentioned in \cite{DK10}, the controlled growth conditions are optimal (see, for instance, a counterexample in \cite{Pala09}) unless some additional boundedness conditions on weak solutions are imposed.

Under the above assumptions, we prove that weak solutions to \eqref{eq3.24} are
{\em globally} H\"{o}lder continuous with H\"{o}lder exponents depending only on the dimension and the integrability of $f$ and $g$.
Indeed, as noted in \cite{Pala09} and \cite{DK10},  we have an explicit description of the H\"{o}lder exponent in terms of $\sigma$ and $\tau$ if $f \in L_\sigma(\Omega)$ and $g \in L_\tau(\Omega)$, $\sigma > d$, $\tau > d/2$, whereas such an explicit H\"{o}lder exponent is not shown in the De Giorgi-Moser-Nash theory.
To obtain the desired regularity, we prove higher integrability of solutions. 
Precisely, we show that a weak solution to \eqref{eq3.24} is a member of $W_p^1(\Omega)$,
where $p>d$ is determined only by $\sigma$ and $\tau$ above.
Then the globally H\"{o}lder continuity of the weak solution follows easily from the Sobolev embedding theorem.
In addition to the fact that the low order terms satisfy the controlled growth conditions, note that in this paper the leading coefficients satisfy only a small BMO condition as functions of $x\in \bR$. Thus they are not necessarily uniformly continuous functions in $x$.
We remark that in general global regularity cannot be expected for systems (see \cite{Gi_ex,StJoMa86}), and even for partial regularities usually one requires the leading coefficients to possess certain regularity in all involved variables (usually uniform continuity).

With the controlled growth conditions, conormal derivative boundary value problems (in other words, Neumann boundary value problems) for quasilinear equations/systems in divergence form have been studied in Arkhipova's papers \cite{Arkh93, Arkh94, Arkh95} (also see the references therein) where she proved reverse H\"{o}lder inequalities and partial regularities of weak solutions up to the boundary. 
In this paper, using a reverse H\"{o}lder inequality as well as $L_p$-theory for linear equations, we show that weak solutions are indeed H\"{o}lder continuous up to the boundary
if the given quasilinear equation has appropriate regularity (not necessarily continuous) on the leading coefficients and the domain is Lipschitz.

When the Dirichlet boundary condition is imposed, Dong and the author established in \cite{DK10} the global H\"{o}lder continuity of weak solutions to equations as in \eqref{eq3.24}
with the same controlled growth conditions when the boundary condition is zero. 
This paper continues to investigate the same type of quasilinear equations, but the boundary condition is of the Neumann type.
That is, we deal with quasilinear divergence type equations with zero conormal derivative boundary value.

In \cite{DK10} we first proved reverse H\"{o}lder inequalities for weak solutions to elliptic and parabolic quasilinear equations, which give slightly better integrability of weak solutions. 
Specifically, for example, a weak solution $u \in W_2^1(\Omega)$ to the elliptic quasilinear equation turns out to be in $W_p^1(\Omega)$ for $p>2$.
The exponent $p$ may not be sufficiently large to give a H\"{o}lder continuity of weak solutions via the Sobolev embedding theorem.  
However, the fact that $p>2$ is enough to give the boundedness and H\"{o}lder continuity of weak solutions by making use of relatively well-known results on divergence type quasilinear equations with zero boundary condition (see \cite{LU, LSU}).
Here the H\"{o}lder continuity is for a uniform continuity of weak solutions, but is not necessarily strong enough to give the desired optimal H\"{o}lder regularity of solutions.
Then using $L_p$-estimates for linear equations, we derive an iteration of $L_p$-estimates, which increases the exponent $p$ until we have sufficient integrability of solutions guaranteeing the global optimal H\"{o}lder regularity of solutions.
As noted in \cite{DK10}, since the reverse H\"{o}lder inequalities are not available for the Dirichlet boundary value problems, in \cite{DK10} we first had to prove reverse H\"{o}lder inequalities for quasilinear elliptic and parabolic equations under the controlled growth conditions.

As to the conormal derivative boundary value problems for quasilinear equations under the controlled growth conditions, as noted above, reverse H\"{o}lder inequalities have already been investigated in \cite{Arkh93, Arkh94, Arkh95} for elliptic and parabolic systems with non-zero conormal derivative boundary values.
Thus, in this paper we concentrate more on the necessary boundedness of solutions as well as a H\"{o}lder continuity for a uniform continuity of solutions.
In fact, similar boundedness and H\"{o}lder continuity results can be found in 
\cite{LU} and \cite{Lieb83}  with possibly different growth conditions.
In particular, \cite[Chapter 10]{LU} shows a H\"{o}lder continuity using a boundary flattening argument when the domain is $C^{1,1}$.
Recently, Winkert studied in \cite{Wink10} the boundedness of weak solutions to quasilinear elliptic equations satisfying natural growth conditions with a conormal derivative boundary condition. The growth conditions correspond to the case with $\lambda_i = 1$ above if weak solutions are in $W_2^1(\Omega)$.
Winkert and Zacher treated in \cite{WZ11} the global boundedness of weak solutions to
the conormal derivative problem for nonlinear elliptic equations where their nonstandard growth conditions cover the {\em strictly} controlled growth conditions. 

We prove the boundedness of weak solutions by making use of the reverse H\"{o}lder inequality (Theorem \ref{thm1}). It is essential to have $u \in W_p^1(\Omega)$, $p > 2$, for a weak solution $u \in W_2^1(\Omega)$ in order to prove the boundedness when the quasilinear equation satisfies the controlled growth conditions. The lines of the proof for the boundedness are based on De Giorgi's iteration technique and similar to those in \cite{LU, Lieb83, WZ11}.
Then we prove a H\"{o}lder continuity of weak solutions by following the argument in \cite{LU}.
Finally, we apply $L_p$-theory, developed in \cite{DongKim08a, DK09}, for linear elliptic equations with conormal derivative boundary conditions when the leading coefficients have small mean oscillations.

We remark that the iteration argument for the repeated use of $L_p$-estimates was previously used by Palagachev in \cite{Pala09}, where he derived the global H\"{o}lder regularity of solutions, as in this paper, by proving higher integrability of solutions. 
The equations considered in \cite{Pala09} are quasilinear elliptic equations with the Dirichlet boundary condition under the strictly controlled growth conditions, and the leading coefficients are in the class of vanishing mean oscillations (VMO). 
Also see \cite{Pala10} and \cite{PalSof11}, where the global H\"{o}lder regularity of solutions to Dirichlet problems on Reifenberg flat domains is discussed when the leading coefficients have small mean oscillations. In \cite{Pala10} the strictly controlled growth conditions are imposed and the existence of solutions is also discussed. In \cite{PalSof11} the controlled growth conditions are imposed on quasilinear elliptic and parabolic equations.

As a final remark, we refer the reader to the paper \cite{DK10} and references therein for more information about various growth conditions and the (partial) regularity of weak solutions to divergence type elliptic and parabolic equations/systems. 

This paper is organized as follows. In Section \ref{main} we introduce our assumptions and main results of this paper.
Then we obtain the boundedness and H\"{o}lder continuity of solutions in Sections \ref{sec3}
and \ref{sec4}, respectively.
In Section \ref{sec5} we present some $L_p$-theory for linear equations which is necessary in the proof of Theorem \ref{thm2} in Section \ref{pfmainth}.
Section \ref{HolderReg} is an independent section describing a function class, functions in which satisfy H\"{o}lder continuity.
Section \ref{elliptic-sec} is devoted to the reverse H\"{o}lder inequality.

\section{Main results}
							\label{main}

For a given function $u=u(x)$ defined on $\Omega \subset \bR^d$, we use
$D_i u$ for $\partial u/\partial x^i$.
For $\alpha \in (0,1]$, we define
\[
|u|_{\alpha,\Omega}=|u|_{0,\Omega}+[u]_{\alpha,\Omega}: = \sup_{x\in \Omega}\,|u(x)|+\sup_{\substack{x, y \in \Omega\\ x\neq y}} \frac{|u(x)-u(y)|}{|x-y|^\alpha}.
\]
By $C^{\alpha}(\Omega)$ we denote the set of all bounded measurable functions $u$ on $\Omega$ for which $|u|_{\alpha,\Omega}$ is finite.
We write $N(d,p,\cdots)$ if $N$ is a constant depending only
on the prescribed quantities $d, p,\cdots$.
Throughout the paper, the domain $\Omega$ satisfies the following Lipschitz condition, where the constant $\beta$ will be specified later. Unless specified otherwise, $\Omega$ is always bounded.

\begin{assumption}[$\beta$]
                                    \label{assump2}
There is a constant $R_0\in (0,1]$ such that, for any $x_0\in \partial\Omega$ and $r\in(0,R_0]$, there exists a Lipschitz
function $\varphi$: $\bR^{d-1}\to \bR$ such that
$$
\Omega\cap B_r(x_0) = \{x \in B_r(x_0)\, :\, x_1 >\varphi(x')\}
$$
and
$$
\sup_{x',y'\in B_r'(x_0'),x' \neq y'}\frac {|\varphi(y')-\varphi(x')|}{|y'-x'|}\le \beta
$$
in an appropriate coordinate system.
\end{assumption}

Let us recall the controlled growth conditions on the lower order terms:
\begin{equation}
								\label{eq110409_01}
|a_i(x,u)| \le \mu_1 (|u|^{\gamma/2} + f),
\quad
|b(x,u,\nabla u)| \le \mu_2 (|\nabla u|^{2(1-1/\gamma)} + |u|^{\gamma-1} + g),	
\end{equation}
where $\mu_1$, $\mu_2$, $\mu_3>0$ are some constants, $\gamma$ is defined as in \eqref{eq110411_04}, and
$$
f\in L_2(\Omega),\quad g\in L_{\frac \gamma {\gamma-1}}(\Omega).
$$

\begin{theorem}[Reverse H\"older inequality]
                                        \label{thm1}
Let $u\in W^1_2(\Omega)$ be a weak solution to \eqref{eq3.24}. Suppose in addition that $f\in L_{\sigma}(\Omega)$ and $g\in L_{\tau}(\Omega)$ for some $\sigma\in (2,\infty)$ and $\tau\in (\gamma/(\gamma-1),\infty)$. Then there exists $p>2$ depending only on $d$, $\mu$, $\mu_1$,
$\mu_2$, $\gamma$, $u$, and $\beta$, such that
$$
\|u\|_{L_{\gamma p/2}(\Omega)} + \|Du\|_{L_p(\Omega)}\le N,
$$
where $N = N(d,\mu,\mu_1,\mu_2,\sigma,\tau, \gamma, u, \|f\|_{L_\sigma(\Omega)},
\|g\|_{L_\tau(\Omega)}, \beta, R_0, \text{\rm diam}\Omega)$.
\end{theorem}

This is proved in \cite{Arkh93} for $d > 2$. Also see \cite{Arkh95} for a linear case with $d > 2$.
For reader's convenience, we give the key proposition (Proposition \ref{prop2}) in Section \ref{elliptic-sec} which readily implies the theorem including the case $d=2$.
As in \cite{Arkh93}, Theorem \ref{thm1} is true for elliptic systems under the same conditions.

To get the optimal global regularity for the equation \eqref{eq3.24}, we need a few more assumptions.
Let
$$
A^{\#}_R = \sup_{1 \le i,j \le d}\sup_{\substack{x_0 \in \bR^d\\ z_0 \in \bR, r \le R}}\dashint_{B_r(x_0)}\dashint_{B_r(x_0)} |A_{ij}(x,z_0) - A_{ij}(y,z_0)| \, dx \, dy.
$$
The following assumption indicates that $A_{ij}(x,\cdot)$ have small mean oscillations as functions of $x\in\bR^d$.

\begin{assumption}[$\rho$]
								\label{SMO}
There is a constant $R_1 \in (0,1]$ such that $A^{\#}_{R_1} \le \rho$.
\end{assumption}

We also need a continuity assumption on $A_{ij}(\cdot,z)$ as functions of $z\in\bR$.

\begin{assumption}
								\label{Aconti}
There exists a continuous nonnegative function $\omega(r)$ defined on $[0,\infty)$
such that $\omega(0) = 0$ and
$$
|A_{ij}(x_0,z_1) - A_{ij}(x_0,z_2)|
\le \omega\left(|z_1 - z_2|\right)
$$
for all $x_0 \in \bR^d$
and $z_1, z_2 \in \bR$.
\end{assumption}

Set
$$
q^* = \left\{
\begin{aligned}
&\frac{qd}{d-q} \quad &\text{if} \quad q < d,\\
&\text{arbitrary large number} > 1 \quad &\text{if} \quad q \ge d.	
\end{aligned}
\right.
$$
Note that if $q<d$, then $1/q^*=1/q-1/d$. We now state the main result of this paper.

\begin{theorem}[Optimal global regularity]
                                        \label{thm2}
Let $u\in W^1_2(\Omega)$ be a weak solution to \eqref{eq3.24}. Suppose in addition that $f\in L_{\sigma}(\Omega)$ and $g\in L_{\tau}(\Omega)$ for some $\sigma\in (d,\infty)$ and $\tau\in (d/2,\infty)$.
Then there exist positive $\beta = \beta(d, \mu, \sigma, \tau)$
and $\rho = \rho(d,\mu, \sigma, \tau)$ such that,
under Assumption \ref{assump2}($\beta$) and Assumption \ref{SMO}($\rho$),
we have
\begin{equation}
								\label{eq110405_01}
\|u\|_{W^1_p(\Omega)} \le N,
\quad
\text{where}
\quad
p=\min\{\sigma,\tau^*\}>d	
\end{equation}
and
$N=N(d,\mu,\mu_1,\mu_2,\sigma,\tau, \gamma, u, \|f\|_{L_\sigma(\Omega)},
\|g\|_{L_\tau(\Omega)}, R_1, \omega, R_0, \text{\rm diam}\Omega)$.
Consequently, we have $u\in C^\alpha(\overline\Omega)$ for $\alpha=1-d/p$.
\end{theorem}

\section{Boundedness of solutions under controlled growth conditions}
							\label{sec3}

In the proof of Theorem \ref{thm2} it is essential to have a H\"{o}lder continuity of weak solutions to \eqref{eq3.24}.
To achieve this,
in this section we prove that the weak solutions are globally bounded.

\begin{lemma}
								\label{lem20110217}
Under the conditions \eqref{eq110409_02} and \eqref{eq110409_01} with $f\in L_{\sigma}(\Omega)$ and $g\in L_{\tau}(\Omega)$ for some $\sigma\in (d,\infty)$ and $\tau\in (d/2,\infty)$, we have
\begin{equation}
								\label{eq2011021702}
\left(A_{ij}\xi_j + a_i\right)\xi_i \ge \frac{\mu}2|\xi|^2 - N |u|^\gamma - N |u|^2 \psi(x),	
\end{equation}
\begin{equation}
								\label{eq2011021501}
|b(x,u,\xi)u| \le \frac{\mu}4 |\xi|^2 + N |u|^\gamma + N |u|^2 \psi(x)	
\end{equation}
for $\xi \in \bR^d$ and $|u| \ge 1$, where $\psi \in L_q(\Omega)$, $q:=\min\{\sigma/2, \tau\} > \frac d 2$, and $N=N(\mu,\mu_1,\mu_2)$. 
\end{lemma}

\begin{proof}
To prove \eqref{eq2011021702}, we first see
$$
|a_i(x,u) \xi_i| \le \mu_1 |\xi| (|u|^{\gamma/2} + f)
\le \varepsilon \mu_1 |\xi|^2 + N(\varepsilon) \mu_1 |u|^\gamma + N(\varepsilon) \mu_1 |f|^2
$$
$$
\le \varepsilon \mu_1 |\xi|^2 + N(\varepsilon) \mu_1 |u|^\gamma + N(\varepsilon)\mu_1 |u|^2 |f|^2,
$$
provided that $|u| \ge 1$.
By taking $\varepsilon = \mu/(2\mu_1)$, we have
$$
\left(A_{ij}(x,u)\xi_j + a_i(x,u)\right)\xi_i
\ge \frac{\mu}2 |\xi|^2 - N(\mu,\mu_1) |u|^\gamma - N(\mu,\mu_1) |u|^2 |f|^2.
$$
Now we take $\psi = |f|^2 + g$.
Then the inequality \eqref{eq2011021702} follows.

For the inequality \eqref{eq2011021501}, we have
\begin{multline*}
|b(x,u,\xi)||u|
\le \mu_2 \left( |u||\xi|^{2(1-1/\gamma)} + |u|^{\gamma} + |u| g \right)
\\
\le \mu_2 \left( \varepsilon |\xi|^2 + N(\varepsilon)|u|^{\gamma} + |u|^2 g\right)
= \frac{\mu}4 |\xi|^2 + N(\mu, \mu_2)|u|^\gamma + N(\mu,\mu_2) |u|^2 g
\end{multline*}
for $|u| \ge 1$.
Upon recalling the definition of $\psi$ we obtain the desired inequality.
\end{proof}

Let $\sigma, \tau$ be numbers satisfying $\sigma\in (d,\infty)$ and $\tau \in (d/2,\infty)$, respectively.
Find $q_1 \in (1,\infty)$ satisfying
\begin{equation}
								\label{eq2011021701}
\frac 1 2 < \frac{1}{q_1} \le \frac{\gamma}4\left(1 - \frac{2(\gamma-2)}{\gamma p}\right),
\quad
\frac 1 2 < \frac{1}{q_1} \le \frac{\gamma}4\left(1 - \frac{1}{\min\{\sigma/2, \tau\}}\right),	
\end{equation}
where $p>2$ is the exponent from Theorem \ref{thm1}.
Indeed, this is possible since $p > 2$ and
$$
\frac 1 2 = \frac{\gamma}4\left(1 - \frac{1}{d/2}\right) < \frac{\gamma}4\left(1 - \frac{1}{\min\{\sigma/2, \tau\}}\right)
\quad
\text{if}
\quad
d > 2.
$$
When $d=2$, we take $\gamma > 2$ so that
$$
\gamma > \frac{2 \min\{\sigma/2, \tau\}}{\min\{\sigma/2, \tau\}-1}.
$$
Note $1<q_1<2$ and $\gamma q_1 > 4$.

\begin{lemma}
								\label{lem20110216_01}
Let $u \in W_2^1(\Omega)$ be a solution to \eqref{eq3.24} and
$f\in L_{\sigma}(\Omega)$, $g\in L_{\tau}(\Omega)$ for some $\sigma\in (d,\infty)$ and $\tau\in (d/2,\infty)$.
Then
$$
\int_{A_k} |\nabla u|^2 \, dx 
\le N \left(\int_{A_k} |u|^{\gamma q_1/2} \, dx\right)^{\frac{4}{\gamma q_1}}
\left(\int_{A_k} |u|^{\frac{\gamma q_1(\gamma-2)}{\gamma q_1 -4}} \, dx\right)^{\frac{\gamma q_1 -4}{\gamma q_1}}
$$
$$
+ \left(\int_{A_k} |u|^{\gamma q_1/2} \, dx\right)^{\frac{4}{\gamma q_1}}
\left(\int_{A_k} \psi^{\frac{\gamma q_1}{\gamma q_1 -4}} \, dx\right)^{\frac{\gamma q_1 -4}{\gamma q_1}},
$$
where $q_1$ is from \eqref{eq2011021701}, $N=N(\mu,\mu_1,\mu_2)$, and
$$
A_k = \{ x \in \Omega: u(x) > k\},
\quad 
k \ge 1,
\quad
\psi = |f|^2 + g.
$$

\end{lemma}

\begin{proof}
First note that by Theorem \ref{thm1}, the definition of $\psi$, and \eqref{eq2011021701}, i.e.,
$$
\frac{\gamma q_1(\gamma-2)}{\gamma q_1 -4}
\le 
\gamma p/2,
\quad
\frac{\gamma q_1}{\gamma q_1 -4}
\le \min\{\sigma/2, \tau\},
$$
we have 
$$
\int_{A_k} |u|^{\frac{\gamma q_1(\gamma-2)}{\gamma q_1 -4}} \, dx < \infty,
\quad
\int_{A_k} \psi^{\frac{\gamma q_1}{\gamma q_1 -4}} \, dx < \infty.
$$

By taking $\phi = (u-k)_+ \in W_2^1(\Omega)$ as a test function, we obtain
$$
\int_{A_k} \left( A_{ij} D_ju D_i u + a_i D_i u \right) \, dx  = \int_{A_k} b (u-k) \, dx.
$$
From Lemma \ref{lem20110217} it follows that
$$
\text{LHS} \ge \frac \mu 2 \int_{A_k} |\nabla u|^2 \, dx - N \int_{A_k} |u|^\gamma \, dx - N \int_{A_k} |u|^2 \psi \, dx,
$$	
and
$$
\text{RHS} \le \int_{A_k} |b||u| \,dx
\le \frac \mu 4 \int_{A_k} |\nabla u|^2 \, dx + N \int_{A_k} |u|^\gamma \, dx 
+ \int_{A_k} |u|^2 \psi \, dx,
$$
where $N=N(\mu,\mu_1,\mu_2)$.
Combining the above two inequalities gives
$$
\int_{A_k} |\nabla u|^2 \, dx 
\le N \int_{A_k} |u|^\gamma \, dx + N\int_{A_k} |u|^2 \psi \, dx.
$$
Then we use H\"{o}lder's inequality to obtain the desired inequality (recall that $\gamma q_1/4 > 1$). That is,
$$
\int_{A_k} |u|^\gamma \, dx
= \int_{A_k} |u|^2 |u|^{\gamma-2} \, dx
\le \left(\int_{A_k} |u|^{2\frac{\gamma q_1}{4}} \, dx\right)^\frac{4}{\gamma q_1}
\left(\int_{A_k} |u|^{\frac{\gamma q_1(\gamma-2)}{\gamma q_1 - 4}} \, dx\right)^\frac{\gamma q_1 - 4}{\gamma q_1},
$$
$$
\int_{A_k} |u|^2 \psi \, dx
\le \left(\int_{A_k} |u|^{2\frac{\gamma q_1}{4}}\,dx \right)^{\frac{4}{\gamma q_1}}
\left(\int_{A_k} \psi^{\frac{\gamma q_1}{\gamma q_1 - 4}} \, dx \right)^\frac{\gamma q_1 - 4}{\gamma q_1}.
$$
\end{proof}

A similar estimate as in the above lemma is needed on the set $\cB_k = \{ x \in \Omega: u(x) < k\}$.

\begin{lemma}
								\label{lem110411_01}
Let $u \in W_2^1(\Omega)$ be a solution to \eqref{eq3.24} and 
$f\in L_{\sigma}(\Omega)$, $g\in L_{\tau}(\Omega)$ for some $\sigma\in (d,\infty)$ and $\tau\in (d/2,\infty)$.
Then
$$
\int_{\cB_k} |\nabla u|^2 \, dx 
\le N \left(\int_{\cB_k} |u|^{\gamma q_1/2} \, dx\right)^{\frac{4}{\gamma q_1}}
\left(\int_{\cB_k} |u|^{\frac{\gamma q_1(\gamma-2)}{\gamma q_1 -4}} \, dx\right)^{\frac{\gamma q_1 -4}{\gamma q_1}}
$$
$$
+ \left(\int_{\cB_k} |u|^{\gamma q_1/2} \, dx\right)^{\frac{4}{\gamma q_1}}
\left(\int_{\cB_k} \psi^{\frac{\gamma q_1}{\gamma q_1 -4}} \, dx\right)^{\frac{\gamma q_1 -4}{\gamma q_1}},
$$
where $q_1$ is from \eqref{eq2011021701}, $N=N(\mu,\mu_1,\mu_2)$, and
$$
\cB_k = \{ x \in \Omega: u(x) < k\},
\quad
k \le -1,
\quad
\psi = |f|^2 + g.
$$
\end{lemma}

\begin{proof}
The proof follows from the lines of the proof for Lemma \ref{lem20110216_01}
with $\phi = (k-u)_+$.
\end{proof}

In the proof of boundedness of weak solutions, we need the following well-known result.
It can also be found, for example, in \cite{LU, LSU} if $\delta_1=\delta_2$.

\begin{lemma}
								\label{lem110411}
Let $\{\Psi_n\}$, $n = 0, 1, 2, \cdots$, be a sequence of positive numbers satisfying
$$
\Psi_{n+1} \le K b^n \left(\Psi_n^{1+\delta_1} + \Psi_n^{1+\delta_2}\right),
\quad
n=0,1,2,\cdots,
$$	
for some $b>1$, $K > 0$, and $\delta_2 \ge \delta_1$.
If
$$
\Psi_0 \le (2K)^{-\frac{1}{\delta_1}}b^{-\frac{1}{\delta_1^2}},
$$
then
$$
\Psi_n \le (2K)^{-\frac{1}{\delta_1}}b^{-\frac{1}{\delta_1^2}-\frac{n}{\delta_1}},
\quad n \in \bN.
$$
Thus, in particular, $\Psi_n \to 0$ as $n \to \infty$.
\end{lemma}

In the following theorem we prove the boundedness of weak solutions to \eqref{eq3.24}
using an iteration argument of De Giorgi type.

\begin{theorem}
								\label{thm110411}
Let $u \in W_2^1(\Omega)$ be a solution to \eqref{eq3.24} and 
$f\in L_{\sigma}(\Omega)$, $g\in L_{\tau}(\Omega)$ for some $\sigma\in (d,\infty)$ and $\tau\in (d/2,\infty)$.
Then for some number $M$, depending only on
$d$, $\mu$, $\mu_1$, $\mu_2$, $\gamma$, $p$, $\sigma$, $\tau$, $u$, $\|f\|_{L_\sigma(\Omega)}$, $\|g\|_{L_\tau(\Omega)}$, $\beta$, $R_0$,
and $\text{\rm diam}\Omega$,
we have
$$
\|u\|_{L_{\infty}(\Omega)} \le M.
$$
Here $p$ is from Theorem \ref{thm1}.
\end{theorem}

\begin{proof}
We take an increasing sequence
$$
k_n = k\left(2 - \frac 1 {2^n} \right),
\quad n = 0, 1, 2, \cdots,
$$
where $k \ge 1$ will be specified later.
Fix $q_1$ so that it satisfies \eqref{eq2011021701}.
Then set
$$
\gamma_* = \frac{\gamma q_1}{2} > 2,
\quad
\Psi_n := \int_{A_{k_n}} (u-k_n)^{\gamma_*} \, dx,
$$
where
$A_{k_n} = \{ x \in \Omega: u(x) > k_n\}$.
Note that $\frac{\gamma}{\gamma_*} = \frac{2}{q_1} > 1$ since $q_1 \in (1,2)$.

Using the fact that $A_{k_{n+1}} \subset A_{k_n}$, we have
$$
\Psi_n = \int_{A_{k_n}} (u-k_n)^{\gamma_*} \, dx \ge \int_{A_{k_{n+1}}} (u-k_n)^{\gamma_*} \, dx
$$
$$
\ge \int_{A_{k_{n+1}}} u^{\gamma_*} \left(1-\frac{k_n}{k_{n+1}}\right)^{\gamma_*} \, dx
\ge \frac{1}{2^{\gamma_* (n+2)}} \int_{A_{k_{n+1}}} u^{\gamma_*} \, dx.
$$
That is,
\begin{equation}
								\label{eq2011021801}
\int_{A_{k_{n+1}}} u^{\gamma_*} \, dx \le 2^{\gamma_*(n+1)} \Psi_n.	
\end{equation}
We also have
\begin{multline}
								\label{eq2011021802}
|A_{k_{n+1}}| \le \int_{A_{k_{n+1}}} \left(\frac{u-k_n}{k_{n+1}-k_n}\right)^{\gamma_*} \, dx
\\
\le \int_{A_{k_n}} \left(\frac{2^{n+1}}{k}\right)^{\gamma_*} (u-k_n)^{\gamma_*} \, dx
= \frac{2^{{\gamma_*}(n+1)}}{k^{\gamma_*}} \Psi_n.
\end{multline}

We now observe that by H\"{o}lder's inequality
\begin{equation}
								\label{eq110411_06}
\Psi_{n+1} = 
\int_{A_{k_{n+1}}} (u-k_{n+1})^{\gamma_*} \, dx
\le \left(\int_{A_{k_{n+1}}} (u-k_{n+1})^\gamma \, dx\right)^{\frac{\gamma_*}{\gamma}}
|A_{k_{n+1}}|^{\frac{\gamma-\gamma_*}{\gamma}}.	
\end{equation}
Note that by the Sobolev embedding theorem,
\begin{multline}
								\label{eq110411_05}	
\left(\int_{A_{k_{n+1}}} (u-k_{n+1})^\gamma \, dx\right)^{1/\gamma}
\\
\le N \left(\int_{\Omega}|\nabla(u-k_{n+1})_+|^2\,dx\right)^{1/2}
+ N \left(\int_{\Omega}|(u-k_{n+1})_+|^2\,dx\right)^{1/2}
\\
\le N\left(\int_{A_{k_{n+1}}}|\nabla u|^2\,dx\right)^{1/2}
+ N \left(\int_{A_{k_{n+1}}}|u-k_{n+1}|^2\,dx\right)^{1/2} :=I_1 + I_2,
\end{multline}
where $N=N(d,\gamma, \beta,R_0, \text{\rm diam}\Omega)$.
To estimate $I_1$ in \eqref{eq110411_05},
we use Lemma \ref{lem20110216_01} and \eqref{eq2011021801} (recall that $\gamma_*=\gamma q_1/2$)
to get
\begin{equation}
								\label{eq110411_07}
\int_{A_{k_{n+1}}} |\nabla u|^2 \, dx 
\le N \left(\int_{A_{k_{n+1}}} u^{\gamma_*} \, dx\right)^{2/\gamma_*}
\le N 2^{2(n+1)} \Psi_n^{\frac{2}{\gamma_*}},	
\end{equation}
where 
$$
N = N\left(\mu, \mu_1, \mu_2, \gamma, p, \sigma, \tau, \int_{\Omega}|u|^{\gamma p/2} \, dx,
\int_{\Omega} \psi^{\min\{\sigma/2,\tau\}} \, dx\right).
$$
Using the facts that $\gamma_* > 2$ and $\Psi_{n+1} \le \Psi_n$, the term $I_2$ in \eqref{eq110411_05}
is estimate as
\begin{multline}
								\label{eq110411_08}
\left(\int_{A_{k_{n+1}}}|u-k_{n+1}|^2\,dx\right)^{1/2}
\le |A_{k_{n+1}}|^{1/2-1/\gamma_*}\left(\int_{A_{k_{n+1}}} (u-k_{n+1})^{\gamma_*} \, dx \right)^{1/\gamma_*}
\\
= |A_{k_{n+1}}|^{1/2-1/\gamma_*}\Psi_{n+1}^{1/\gamma_*}
\le |A_{k_{n+1}}|^{1/2-1/\gamma_*} \Psi_n^{1/\gamma_*}.
\end{multline}
Combining \eqref{eq110411_06}, \eqref{eq110411_05}, \eqref{eq110411_07}, \eqref{eq110411_08},
and \eqref{eq2011021802},
we obtain
$$
\Psi_{n+1} \le N |A_{k_{n+1}}|^{\frac{\gamma-\gamma_*}{\gamma}}
\left[2^{\gamma_* (n+1)} \Psi_n + |A_{k_{n+1}}|^{\frac{\gamma_*-2}{2}} \Psi_n \right]
$$
$$
\le N \left(\frac{2^{{\gamma_*}(n+1)}}{k^{\gamma_*}} \Psi_n\right)^{\frac{\gamma-\gamma_*}{\gamma}} 2^{\gamma_* (n+1)} \Psi_n
+ N \left(\frac{2^{{\gamma_*}(n+1)}}{k^{\gamma_*}} \Psi_n\right)^{\frac{\gamma_*(\gamma-2)}{2\gamma}} \Psi_n =: J_1 + J_2,
$$
where
$$
J_1 = N k^{-\frac{\gamma_*(\gamma-\gamma_*)}{\gamma}} 2^{\gamma_*\left(\frac{\gamma-\gamma_*}{\gamma}+1\right)} 
2^{\gamma_*\left(\frac{\gamma-\gamma_*}{\gamma}+1\right)n}
\Psi_n^{1+\frac{\gamma-\gamma_*}{\gamma}},
$$
$$
J_2 = N k^{-\frac{\gamma_*^2(\gamma-2)}{2\gamma}}
2^{\frac{\gamma_*^2(\gamma-2)}{2\gamma}}2^{\frac{\gamma_*^2(\gamma-2)}{2\gamma}n}
\Psi_n^{1+\frac{\gamma_*(\gamma-2)}{2\gamma}}.
$$

Set
$$
\delta_1 = \frac{\gamma-\gamma_*}{\gamma},
\quad
\delta_2 = \frac{\gamma_*(\gamma-2)}{2\gamma},
\quad
b = \max\left\{2^{\gamma_*\left(\frac{\gamma_*(\gamma-2)}{2\gamma}+1\right)},
2^{\frac{\gamma_*^2(\gamma-2)}{2\gamma}}\right\},
$$
$$
K = N k^{-\frac{\gamma_*(\gamma-\gamma_*)}{\gamma}} b.
$$
Then $\delta_2 \ge \delta_1 > 0$, $b>1$, and
$$
J_1 \le K b^n \Psi_n^{1+\delta_1},
\quad
J_2 \le K b^n \Psi_n^{1+\delta_2}.
$$
Hence
$$
\Psi_{n+1} \le K b^n \left(\Psi_n^{1+\delta_1} + \Psi_n^{1+\delta_2}\right).
$$

Observe that
$$
\Psi_0 = \int_{u > k}(u-k)^{\gamma_*}\, dx
\le \int_{\Omega} u_+^{\gamma_*} \, dx
$$
$$
= \left( (2K)^{\frac{1}{\delta_1}} b^{\frac{1}{\delta_1^2}} \int_{\Omega} u_+^{\gamma_*} \, dx \right) (2K)^{-\frac{1}{\delta_1}} b^{-\frac{1}{\delta_1^2}}
\le (2K)^{-\frac{1}{\delta_1}} b^{-\frac{1}{\delta_1^2}}
$$
provided that
$$
(2K)^{\frac{1}{\delta_1}} b^{\frac{1}{\delta_1^2}} \int_{\Omega} u_+^{\gamma_*} \, dx \le 1,
$$
that is, if we take $k \ge 1$ so that
$$
k = \max\left\{1, 2^{\frac{1}{\gamma_*\delta_1}} N^{\frac{1}{\gamma_*\delta_1}}  b^{\frac{1+\delta_1}{\gamma_*\delta_1^2}} \left(\int_{\Omega} u_+^{\gamma_*} \, dx\right)^{1/\gamma_*}\right\}.
$$
Then by Lemma \ref{lem110411} it follows that $u \le 2k$ on $\Omega$.
To prove that $u$ is bounded below, we repeat the above argument using Lemma \ref{lem110411_01}.
The theorem is proved.
\end{proof}

\section{H\"{o}lder continuity}
							\label{sec4}

The inequality \eqref{eq2011021702} holds true for $|u| \ge 1$.
However, from the proof of Lemma \ref{lem20110217} it is possible to have
\begin{equation}
								\label{eq110412_03}
\left(A_{ij}D_j u + a_i\right)D_i u \ge \frac{\mu}2|\nabla u|^2 -  N|u|^\gamma - N |f|^2	
\end{equation}
for all values of $u$, where $N=N(\mu,\mu_1)$.
Observe that from the condition \eqref{eq110409_01} on $b$, we obtain
$$
|b(x,u,\nabla u)(u-k)_+| \le \mu_2 (u-k)_+ \left(|\nabla u|^{2(1-1/\gamma)} + |u|^{\gamma-1} + g\right)
$$
$$
\le \frac{\mu}{4}|\nabla u|^2 + N(\mu,\mu_2) (u-k)_+^\gamma + \mu_2(u-k)_+\left(|u|^{\gamma-1} + g\right)
$$
$$
\le \frac{\mu}{4}|\nabla u|^2 + N\left(1 + |u|^{\gamma-1} + g\right)
$$
for $k \ge u - 1$, where $N=N(\mu,\mu_2)$. By the same reasoning
$$
|b(x,u,\nabla u)(k-u)_+| \le \frac{\mu}{4}|\nabla u|^2 + N\left(1 + |u|^{\gamma-1} + g\right)
$$
for $k \le u+1$.
Set
\begin{equation}
								\label{eq0328_01}
\varphi_0 = |u|^\gamma + |f|^2,
\quad
\varphi_1 = |u|^{\gamma/2} + f,
\quad
\varphi_2 =  |u|^{\gamma-1} + g + 1.	
\end{equation}
Then by \eqref{eq110412_03} and the condition on $a_i$ we have
\begin{equation}
								\label{eq110412_02}
(A_{ij}D_ju + a_i)D_i u \ge \frac \mu 2 |\nabla u|^2 - N \varphi_0,
\quad
|a_i| \le N \varphi_1
\end{equation}
for all values of $u$, where $N=N(\mu,\mu_1)$.
We also have
\begin{equation}
								\label{eq110412_04}
|b(x,u,\nabla u) (u-k)_+ | \le \frac \mu 4 |\nabla u|^2 + N \varphi_2
\quad
\text{for}
\quad
k \ge u - 1,	
\end{equation}
$$
|b(x,u,\nabla u) (k-u)_+ | \le \frac \mu 4 |\nabla u|^2 + N \varphi_2
\quad
\text{for}
\quad
k \le u + 1,
$$
where $N=N(\mu,\mu_2)$.
As shown in Theorem \ref{thm110411}, $|u| \le M$ on $\Omega$ for some constant $M$.
Thus, if $f \in L_\sigma(\Omega)$ and $g \in L_\tau(\Omega)$ for some $\sigma \in (d,\infty)$ and $\tau \in (d/2,\infty)$ as in Theorem \ref{thm2}, we have
$$
\varphi_0, \varphi_2 \in L_q(\Omega)
\quad
\varphi_1 \in L_{2q}(\Omega),
\quad
q = \min\{\sigma/2,\tau\} > d/2.
$$

\begin{lemma}
Let $\zeta \in W_2^1(\bR^d)$ have a  compact support,
and $k$ be a real number such that $k \ge u - 1$ on the support of $\zeta$.
Let $u \in W_2^1(\Omega)$ be a solution to \eqref{eq3.24} and 
$f\in L_{\sigma}(\Omega)$, $g\in L_{\tau}(\Omega)$ for some $\sigma\in (d,\infty)$ and $\tau\in (d/2,\infty)$.
Then we have
\begin{multline}
								\label{eq110412_06}
\int_{\{u>k\}\cap\Omega} |\zeta \nabla u|^2 \, dx
\le N \int_{\{u>k\}\cap\Omega} |\zeta_x|^2 (u-k)^2 \, dx
\\
+ N \left(\int_{\{u>k\}\cap\Omega} \zeta^{2\bar q} \, dx\right)^{1/\bar q}
\left( \int_{\{u>k\}\cap\Omega} \left(\varphi_0 + \varphi_1^2 + \varphi_2\right)^q \, dx \right)^{1/q},	
\end{multline}
where $\varphi_i$, $i=0,1,2$, are those in \eqref{eq0328_01}, $q = \min\{\sigma/2, \tau\}$, $1/\bar q + 1/q = 1$, and $N=N(\mu,\mu_1,\mu_2)$.

Now if $k \le u + 1$ on the support of $\zeta$, then
\begin{multline}
									\label{eq110413_20}
\int_{\{u<k\}\cap\Omega} |\zeta \nabla u|^2 \, dx
\le N \int_{\{u<k\}\cap\Omega} |\zeta_x|^2 (k-u)^2 \, dx
\\
+ N \left(\int_{\{u<k\}\cap\Omega} \zeta^{2\bar q} \, dx\right)^{1/\bar q}
\left( \int_{\{u<k\}\cap\Omega} \left(\varphi_0 + \varphi_1^2 + \varphi_2\right)^q \, dx \right)^{1/q}.	
\end{multline}
\end{lemma}

\begin{proof}
Using $\phi = \zeta^2 (u-k)_+ \in W_2^1(\Omega)$ as a test function,
we obtain
\begin{equation}
								\label{eq110412_01}
\int_{\Omega} \left(A_{ij}(x,u) D_j u + a_i(x,u)\right) \phi_{x_j} \, dx
= \int_{\Omega} b(x,u,\nabla u) \phi \, dx,
\end{equation}
which is equal to
$$
\int_{\{u>k\}\cap\Omega} \left(A_{ij} D_j u + a_i\right)\zeta^2 D_ju \, dx
= \int_{\{u>k\}\cap\Omega} b \zeta^2 (u-k) \, dx
$$
$$
- \int_{\{u>k\}\cap\Omega} \left(A_{ij} D_j u + a_i\right)2 \zeta \zeta_{x_j} (u-k) \, dx.
$$
Note that by \eqref{eq110412_02}
$$
\int_{\{u>k\}\cap\Omega} \zeta^2 \left(A_{ij} D_j u + a_i\right) D_ju \, dx
\ge \frac{\mu}{2}\int_{\{u>k\}\cap\Omega}|\zeta \nabla u|^2 \, dx
 - N \int_{\{u>k\}\cap\Omega} \zeta^2 \varphi_0 \, dx
$$
and
$$
- \int_{\{u>k\}\cap\Omega} \left(A_{ij} D_j u + a_i\right)2 \zeta \zeta_{x_j} (u-k) \, dx
\le \frac{\mu}{8}\int_{\{u>k\}\cap\Omega} |\zeta \nabla u|^2 \,dx
$$
$$
+ N \int_{\{u>k\}\cap\Omega} \left(|\zeta_x|^2 (u-k)^2 + \zeta^2 \varphi_1^2\right)\,dx.
$$
Since $k \ge u-1$ on the support of $\zeta$,
by \eqref{eq110412_04}
$$
\int_{\{u>k\}\cap\Omega} b \zeta^2 (u-k) \, dx
\le \frac{\mu}{4} \int_{\{u>k\}\cap\Omega} |\zeta \nabla u|^2 \, dx 
+ N \int_{\{u>k\}\cap\Omega}\zeta^2\varphi_2\,dx.
$$

Hence \eqref{eq110412_01} is written as
$$
\int_{\{u>k\}\cap\Omega} |\zeta \nabla u|^2 \, dx
$$
$$
\le N \int_{\{u>k\}\cap\Omega} |\zeta_x|^2(u-k)^2 \, dx
+ N \int_{\{u>k\}\cap\Omega} \zeta^2 \left(\varphi_0 + \varphi_1^2 + \varphi_2\right)\,dx,
$$
where $N=N(\mu,\mu_1,\mu_2)$.
Finally, by applying H\"{o}lder's inequality we obtain the desired inequality in the lemma.
The second assertion follows by the same reasoning as above with $\phi = \zeta^2 (k-u)_+$.
\end{proof}

\begin{proposition}
								\label{prop110411}
Let $u \in W_2^1(\Omega)$ be a solution to \eqref{eq3.24} and 
$f\in L_{\sigma}(\Omega)$, $g\in L_{\tau}(\Omega)$ for some $\sigma\in (d,\infty)$ and $\tau\in (d/2,\infty)$.
Then $u \in C^{\alpha_0}(\overline\Omega)$ and
$$
|u|_{\alpha_0,\Omega}
\le N,
$$
where $\alpha_0 \in (0,1)$ and $N$ depend only on 
the parameters for the bound of $u$ in Theorem \ref{thm110411}.
\end{proposition}

\begin{proof}
Let $r_0 > 0$ and $B_{r_0} \subset \Omega$.
For $r \le r_0$ and $\delta \in (0,1)$, let $B_r$ and $B_{r(1-\delta)}$ be balls concentric with $B_{r_0}$.
Let $\zeta$ be an infinitely differentiable function such that
$0 \le \zeta \le 1$,
$\zeta = 1$ on $B_{r(1-\delta)}$
and $\zeta=0$ outside $B_r$.
We may assume that $|D \zeta| \le 1/(\delta r)$.
Then from \eqref{eq110412_06} we obtain
\begin{equation}
								\label{eq110413_01}
\int_{\{u > k\} \cap B_{r(1-\delta)}} |\nabla u|^2 \, dx
\le C \left(\frac{1}{\delta^2 r^{2-d/q}} \max_{\{u > k\} \cap B_r} (u-k)^2 + 1\right)
|\{u > k\} \cap B_r|^{1-1/q}	
\end{equation}
for $B_r \subset \Omega$ and $k \ge \max_{B_r} u - 1$,
where $\delta \in (0,1)$, $q=\min\{\sigma/2,\tau\} > d/2$, and the constant $C$ depends only on the parameters for the bound of $u$ in Theorem \ref{thm110411}.
Using \eqref{eq110413_20} we also obtain \eqref{eq110401_01} for $u$.
Hence $u \in H(\Omega,M,C,1,q)$ in Definition \ref{def110401} when $\Omega = B_{r_0}$.
Therefore, we have the oscillation estimate in Theorem \ref{thm110413}, which indeed implies
\begin{equation}
								\label{eq110413_02}
\text{\rm osc}_{B_r} u
\le N \left(\frac{r}{r_0}\right)^\alpha 
\text{\rm osc}_{B_{r_0}} u  + N r_0^{\alpha_1} r^\alpha
\end{equation}
for all $r \le r_0$, where $\alpha>0$, $\alpha_1>0$, and $N>0$
depend only on $d$, $C$, and $q$.

Let $x_0 \in \partial \Omega$
and $r_0 < R_0$, where $R_0$ is from Assumption \ref{assump2}.
Without loss of generality we assume that $x_0 = 0$ and $\varphi(0) = 0$,
where $\varphi$ is a Lipschitz function 
such that $\Omega_{R_0} = \Omega \cap B_{R_0}
= \{x \in B_{R_0}: x_1 > \varphi(x')\}$.
Under this assumption, since $|\varphi(x')| \le \beta |x'|$,
we observe that
\begin{equation}
								\label{eq110413_03}
\begin{aligned}
\Phi(B_r) \subset \Omega_{r_\beta}
\quad
&\text{for}
\quad 
r < \frac{R_0}{\sqrt{2(1+\beta^2)}},\\
\Phi^{-1}(\Omega_r)
\subset B_{r_\beta}
\quad
&\text{for}
\quad 
r < \frac{R_0}{2(1+\beta^2)},
\end{aligned}
\end{equation}
where $\Phi(y) = (y_1+\varphi(y'),y')$ and $r_\beta = r \sqrt{2(1+\beta^2)}$.
Set $v(y) := u(\Phi(y))$
and
$r_1 := \frac{r_0}{\sqrt{2(1+\beta^2)}}$.
From \eqref{eq110413_03} we have
$v \in W_2^1(B_{r_1}^+)$.
For $r \in (0, r_1]$, let $\psi$ be an infinitely differentiable function such that
$0 \le \psi \le 1$, $\psi = 0$ on $B_{r(1-\delta)}$, and $\psi = 0$ outside $B_r$.
We may assume that $|D\psi| \le 1/(\delta r)$.
By using
$$
\zeta(x) = \psi(\Phi^{-1}(x)),
$$
we obtain from \eqref{eq110412_06}
$$
\int_{\{u>k\}\cap\Omega} |\zeta \nabla u|^2 \, dx
\le N \int_{\{u>k\}\cap\Omega} |\zeta_x|^2 (u-k)^2 \, dx
+ C \left(\int_{\{u>k\}\cap\Omega} \zeta^{2\bar q} \, dx\right)^{1/\bar q}
$$
for $k \ge u - 1$ on the support of $\zeta$,
where $C$ is a constant as in \eqref{eq110413_01}.
By the change of variables, this turns into
$$
\int_{\{v > k\} \cap \Omega_{r(1-\delta)}} |\nabla v|^2 \, dx
\le C \left(\frac{1}{\delta^2 r^{2-d/q}} \max_{\{v > k\} \cap \Omega_r} (v-k)^2 + 1\right)
|\{v > k\} \cap \Omega_r|^{1-1/q},
$$
where $\Omega_r = B_r^+ = \{y \in \bR^d: |y| < r, y_1 > 0\}$.
Similarly, the inequality \eqref{eq110401_01} is proved for $v$.
Hence $v \in H(\Omega,M,C,1,q)$ in Definition \ref{def110401} when $\Omega = B_{r_1}^+$.
Therefore, by Theorem \ref{thm110413}
we have
$$
\text{\rm osc}_{B_r^+} v
\le N \left(\frac{r}{r_1}\right)^\alpha
$$
for all $r \le r_1$, where $\alpha = \alpha(d,C,q)$
and $N = N(d,C,q, r_1, M)$.
This together with the definition of $v$ and \eqref{eq110413_03} shows that, for any $r < \frac{r_0}{2(1+\beta^2)}=:r_2$,
$$
\text{\rm osc}_{\Omega_r} u
\le \text{\rm osc}_{B_{r_\beta}^+} v
\le N \left(\frac{r}{r_2}\right)^\alpha.
$$
Finally, we use this inequality and \eqref{eq110413_02} to finish the proof (for details, see Theorem 8.29 in \cite{GT}).
\end{proof}

\section{$L_p$-estimates for linear equations}
                                \label{sec5}
In the proof of Theorem \ref{thm2} where we prove the global H\"{o}lder regularity result,
it is essential to use some results from $L_p$-theory for linear elliptic equations.
In this section, we consider the linear equation
\begin{equation}
                                \label{eq22.18}
\left\{
  \begin{aligned}
    - D_i(a_{ij} D_j v) + \lambda v &= D_i h_i + h \quad \text{in} \,\, \Omega, \\
    \left(a_{ij}D_j v + h_i\right) \nu_i&= 0 \quad \text{on} \,\, \partial\Omega,
  \end{aligned}
\right.
\end{equation}
where $\nu$ is the outward normal vector to the surface $\partial \Omega$ and $\lambda > 0$, 
and present some $L_p$-solvability as well as $L_p$-estimates necessary to the proof of Theorem \ref{thm2}.

We assume that the leading coefficients $a_{ij}$ have small mean oscillations with respect to $x \in \bR^d$. To describe this assumption, we set
$$
a_R^{\#} = \sup_{1 \le i,j \le d}
\sup_{\substack{x_0 \in \bR^d\\r \le R}}\dashint_{B_r(x_0)}\dashint_{B_r(x_0)} |a_{ij}(x) - a_{ij}(y)| \, dx \, dy.
$$
Assume that $|a_{ij}(x)| \le \mu^{-1}$ and $a_{ij}(x)\xi_i\xi_j \ge \mu |\xi|^2$
for all $\xi \in \bR^d$ and $x \in \bR^d$.
Also we assume
\begin{assumption}[$\rho_1$]
								\label{liSMO}
There is a constant $R_1 \in (0,1]$ such that $a_{R_1}^{\#} \le \rho_1$.
\end{assumption}

We use the following result from \cite{DongKim08a}.
By a half space, we mean, for example, $\bR^d_+=\{x\in\bR^d:x_1>0\}$.
\begin{proposition}
								\label{prop7.1}
Let $\Omega$ be the whole space $\bR^d$, a half space, or a bounded Lipschitz domain.
Let $\hat{p} \in  (2,\infty)$, $p\in [\hat{p}/(\hat{p}-1), \hat{p}]$, $h_i \in L_p(\Omega)$, and $h \in L_p(\Omega)$.
\begin{enumerate}
\item If $\Omega$ is the whole space $\bR^d$ or a half space $\bR^d_+$,
there exists a positive $\rho_1=\rho_1(d,\mu,\hat{p})$
such that, under Assumption \ref{liSMO} ($\rho_1$),
there is a unique $v \in W_p^1(\Omega)$
satisfying \eqref{eq22.18} and
\begin{equation}
                                    \label{eq22.20b}
\sqrt{\lambda}\|v_x\|_{L_p(\Omega)}
+ \lambda \|v\|_{L_p(\Omega)} \le N \sqrt{\lambda}\|h_i\|_{L_p(\Omega)}
+ N \|h\|_{L_p(\Omega)}
\end{equation}
provided that $\lambda \ge \lambda_0$,
where $N > 0$ and $\lambda_0 > 0$ are constants depending only on $d$, $\mu$, $\hat{p}$, 
and  $R_1$. 

\item If $\Omega$ is a bounded Lipschitz domain,
there exist positive $\beta=\beta(d,\mu,\hat{p})$ and $\rho_1=\rho_1(d,\mu,\hat{p})$
such that, under Assumption \ref{assump2} ($\beta$) and Assumption \ref{liSMO} ($\rho_1$),
there is a unique $v \in W_p^1(\Omega)$
satisfying \eqref{eq22.18} and \eqref{eq22.20b}
provided that $\lambda \ge \lambda_0$,
where $N > 0$ and $\lambda_0 \ge 0$ are constants depending only on $d$, $\mu$, $\hat{p}$, $R_0$, $R_1$, and $\text{\rm diam}\Omega$.
\end{enumerate}
\end{proposition}

The proposition above was proved in \cite{DongKim08a} so that the choices of $\beta$ and $\rho_1$ may be different depending on $p$. Also see \cite{DK09}. To find uniform $\rho_1$ and $\beta$ for all $p \in [\hat{p}/(\hat{p}-1),\hat{p}]$,
we use the cited result and an interpolation argument as in \cite{DK10}.
Indeed, if we have the $W_{\hat{p}}^1$
solvability of \eqref{eq22.18} for some $a_{ij}$ and $\Omega$, by the duality, the $W^1_{\hat{p}/(\hat{p}-1)}$ solvability follows. Then we apply Marcinkiewicz's  theorem to get the $W^1_p$ solvability for any $p\in [\hat{p}/(\hat{p}-1),\hat{p}]$.
	
By using Proposition \ref{prop7.1}, we derive the following theorem, where $h$ may have less integrability than those in Proposition \ref{prop7.1}.
Again, the constants $\beta$ and $\rho_1$ are found independent of $\sigma$ and $q$ as long as $\sigma$ and $q^*$ are in an a prior fixed interval.
Recall the definition of $q^*$ given above Theorem \ref{thm2}.

\begin{theorem}
								\label{liRes}
Let $\Omega$ be a bounded Lipschitz domain, $\sigma, q \in (1,\infty)$, $\hat{p} \in (2,\infty)$, $h_i \in L_{\sigma}(\Omega)$, and $h \in L_q(\Omega)$.
Assume that $\sigma, q^* \in [\hat{p}/(\hat{p}-1), \hat{p}]$.
Then there exist positive $\beta=\beta(d,\mu, \hat{p})$ and $\rho_1=\rho_1(d,\mu,\hat{p})$
such that, under Assumption \ref{assump2} ($\beta$) and Assumption \ref{liSMO} ($\rho_1$),
there is a unique $v \in W_p^1(\Omega)$
satisfying \eqref{eq22.18} and
$$%\begin{equation}
%                                    \label{eq22.20}
\|v\|_{W_p^1(\Omega)} \le N \|h_i\|_{L_\sigma(\Omega)}
+ N \|h\|_{L_q(\Omega)},
$$%\end{equation}
provided that $\lambda \ge \lambda_0$,
where $p:= \min \{ \sigma, q^* \}$ and
\begin{equation}
							\label{eq110509}							
\lambda_0 = \lambda_0(d, \mu, \hat{p}, R_0, R_1, \text{\rm diam}\Omega) \ge 0,
\,\,
N = N(d,\mu, \hat{p}, \sigma, q, q^*, R_0, R_1, \lambda, \text{\rm diam}\Omega) > 0.	
\end{equation}
\end{theorem}

\begin{proof}
We split the equation \eqref{eq22.18} into two linear equations with $h \equiv 0$ and $h_i \equiv 0$, $i=1,\cdots,d$, respectively.
Since $h_i \in L_p(\Omega) \subset L_\sigma(\Omega)$ and $p \in [\hat{p}/(\hat{p}-1), \hat{p}]$, by Proposition \ref{prop7.1}, 
we have constants $\beta$ and $\rho_1$, depending only on $d$, $\mu$, and $\hat{p}$,
such that, under Assumption \ref{assump2} ($\beta$) and Assumption \ref{liSMO} ($\rho_1$), there exists a unique solution $v \in W_p^1(\Omega)$ to the equation \eqref{eq22.18} with $h \equiv 0$ satisfying
$$
\|v\|_{W_p^1(\Omega)} \le N \|h_i\|_{L_\sigma(\Omega)},
$$
provided that $\lambda \ge \lambda_0$, where $\lambda_0$ and $N$ depend only on the parameters in \eqref{eq110509}.

Now let $h_i \equiv 0$, $i=1,\cdots,d$.
Thanks to the localization argument using a partition of unity, it is enough to show the uniqueness solvability in $W_{q^*}^1(\Omega)$ of the equation \eqref{eq22.18} along with the following estimate
when $\Omega = \bR^d$ and $\Omega = \bR^d_+$:
\begin{equation}
							\label{eq110509_02}
\|v\|_{W_{q^*}^1(\Omega)} \le N \|h\|_{L_q(\Omega)}.
\end{equation}

In case $\Omega = \bR^d$,
since $h \in L_q(\Omega)$, we find a unique solution $w \in W^2_q(\bR^d)$ to the equation
$$
-\Delta w + \lambda w = h \quad \text{in} \quad \Omega
$$
satisfying
$$
\|w\|_{W_q^2(\Omega)}
\le N \|h\|_{L_q(\Omega)},
$$
where $\lambda > 0$ and $N=N(d,q,\lambda)$.
From the above inequality and the Sobolev imbedding theorem, we know that $w \in W^1_{q^*}(\Omega)$ and
\begin{equation}
                                \label{eq22.10}
\|w\|_{W_{q^*}^1(\Omega)}\le N%(d,q, q^*?)
\|h\|_{L_q(\Omega)}.
\end{equation}
Since $q^* \in [\hat{p}/(\hat{p}-1),\hat{p}]$, by Proposition \ref{prop7.1} we have $\rho_1 = \rho_1(d, \mu,\hat{p})>0$ such that, under Assumption \ref{liSMO} ($\rho_1$), there is a unique solution $\hat w \in W^1_{q^*}(\Omega)$ to the equation
$$
-D_i(a_{ij}D_j \hat w) + \lambda \hat w =D_i((a_{ij}-\delta_{ij})D_j w) \quad \text{in}
\quad \Omega
$$
satisfying
\begin{equation}
							\label{eq110509_01}
\|\hat w\|_{W^1_{q*}(\Omega)}\le N\|Dw \|_{L_{q^*}(\Omega)}, %\le N\|h\|_{L_q(\Omega)},
\end{equation}
provided that $\lambda \ge \lambda_0$,
where $\lambda_0 = \lambda_0(d,\mu,\hat{p},R_1)$ and $N= N(d,\mu,\hat{p}, R_1, \lambda)$.
Clearly $v:=w +\hat w\in W^1_{q^*}(\Omega)$ is a unique solution to \eqref{eq22.18} when $h_i \equiv 0$, $i=1,\cdots,d$. By \eqref{eq110509_01} and \eqref{eq22.10} the solution $v$ satisfies \eqref{eq110509_02}.

In the case that $\Omega=\bR^d_+$, let $w$ be the unique $W^2_q(\bR^d)$ solution to $-\Delta w + \lambda w =\bar h$ in $\bR^d$, where $\bar h$ is the even extension of $h$ with respect to $x_1$. Clearly $w_{x_1}=0$ on $\partial \bR^d$, and as before we know that $w\in W^1_{q^*}(\Omega)$ and satisfies \eqref{eq22.10}. Now we argue as in the previous case. In particular, note that $v:=w+\hat w$ satisfies the boundary condition in \eqref{eq22.18}.
\end{proof}

\begin{remark}
A Dirichlet problem version of the above theorem is proved in \cite{Pala09}, where an $F \in L_{q^*}(\Omega)$ satisfying $\Div F = h$ is found directly, thanks to the Dirichlet boundary condition, by using a Newtonian potential.
Here, since we have the conormal derivative boundary condition, the argument in \cite{Pala09} is not applicable. Instead, we have gone through the interior estimates (when $\Omega = \bR^d$), the boundary estimates (when $\Omega = \bR^d_+$), and the well-known partition of unity argument.
In the above theorem as well as Proposition \ref{prop7.1} for a bounded Lipschitz domain $\Omega$, the $\lambda_0$ can be made equal to zero. (If $\lambda=0$ in \eqref{eq22.18} we need $\int_\Omega h \, dx = 0$.) See Section 7 in \cite{DongKim08a}.
However, we do not pursue this direction here.
\end{remark}

\section{Proof of Theorem \ref{thm2}}
							\label{pfmainth}

Under the assumptions in Theorem \ref{thm2},
Proposition \ref{prop110411} says that $u$ is globally H\"{o}lder continuous on $\Omega$.
Then one can have an extension $\bar u (x)$ on $\bR^d$ of $u(x)$ such that $\bar u (x)$ is H\"{o}lder continuous on $\bR^d$ with the same H\"{o}lder exponent.
Now we define
$$
a_{ij}(x) := A_{ij}(x,\bar u(x)).
$$
Also define
$$
h_i(x) := a_i(x,u(x)),
\quad
h(x) := b(x,u(x),\nabla u(x)).
$$
Then the equation \eqref{eq3.24} turns into
\begin{equation}							 \label{Linearized}
\left\{
  \begin{aligned}
    - D_i(a_{ij} D_ju) &= h_i(t,x)) + h(t,x) \quad \text{in} \quad \Omega, \\
    (a_{ij}D_jv + h_i) \nu_i &= 0 \quad \text{on} \quad \partial \Omega,
  \end{aligned}
\right.
\end{equation}
where $\nu$ is the outward normal vector to the surface $\partial \Omega$.
Note that
\begin{equation}
								\label{eq0212}
|h_i(x)| \le \mu_1( |u|^{\gamma/2} + f)
\le \mu_1( M + f) \in L_{\sigma}(\Omega),
\end{equation}
where $M$ is from Theorem \ref{thm110411},
and
\begin{equation}
								\label{eq0211}
|h(x)| \le \mu_2( |\nabla u|^{2(1-1/\gamma)}+|u|^{\gamma-1} + g).	
\end{equation}
The coefficients $a_{ij}$ in \eqref{Linearized} satisfy, for any $x_0 \in \bR^d$,
$$
\dashint_{x,y\in B_r(x_0)}
| a_{ij}(x) - a_{ij}(y) | \, dx \, dy
$$
$$
= \dashint_{x,y \in B_r(x_0)}
| A_{ij}(x,\bar u(x)) - A_{ij}(y,\bar u(y)) | \, dx \, dy
$$
$$
\le \dashint_{x\in B_r(x_0)}
| A_{ij}(x,\bar u(x)) - A_{ij}(x,\bar u(x_0)) | \, dx
$$
$$
+ \dashint_{x,y \in B_r(x_0)}
| A_{ij}(x, \bar u(x_0)) - A_{ij}(y,\bar u(x_0)) | \, dx \, dy
$$
$$
+ \dashint_{y \in B_r(x_0)}
| A_{ij}(y, \bar u(x_0)) - A_{ij}(y,\bar u(y)) | \, dy
\le 2\omega(N r^{\alpha_0}) + A^{\#}_r,
$$
where the last inequality is due to Assumption \ref{SMO} and Proposition \ref{prop110411}.
That is, by using the notation in Section \ref{sec5}, we have
$$
a_R^{\#} \le 2\omega(N R^{\alpha_0})
+ A^{\#}_R.
$$
Then by Assumptions \ref{SMO} and \ref{Aconti} there exists $R_2\in (0,R_1]$ such that
\begin{equation}
								\label{eq0224}
a_{R_2}^{\#} \le 2\rho,
\end{equation}
where $R_2$ depends on the function $\omega$.

\begin{proof}[Proof of Theorem \ref{thm2}]
We set $\hat{p}$ to be $\max\{\sigma,\tau^*\}$,
and fix
\begin{equation}
								\label{eq0223}
\beta = \beta(d, \hat{p}, \mu),
\quad
\rho= \frac 1 2 \rho_1(d,\hat{p},\mu),								
\end{equation}
where $\beta(d,\hat{p},\mu)$ and $\rho_1(d,\hat{p},\mu)$ are those in Theorem \ref{liRes}.
Also fix $\lambda \ge \lambda_0$, where $\lambda_0=\lambda_0(d,\mu,\hat{p},R_0,R_1,\text{\rm diam}\Omega)$
is taken from Theorem \ref{liRes}.

By Theorem \ref{thm1} there exists $p_0 > 2$ such that
$u \in W_{p_0}^1(\Omega)$.
If $p_0 \ge \min\{\sigma, \tau^*\}$, we immediately obtain \eqref{eq110405_01}.
Otherwise, we see that $u$ satisfies \eqref{Linearized}. By \eqref{eq0212} and \eqref{eq0211},
$h_i \in L_{\sigma}(\Omega)$ and $h \in L_{q_1}(\Omega)$, where
$$
q_1 = \min \left\{ \frac{\gamma}{2(\gamma-1)}p_0, \tau \right\}.
$$
By taking $\left(\frac{\gamma}{2(\gamma-1)}p_0\right)^*$ to be $\tau$ in the case that  
$$
\frac{\gamma}{2(\gamma-1)}p_0 \ge d,
$$
we see that
$$
2 < p_0 < q_1^* \le \tau^*.
$$
Indeed, it is easily verified because
\begin{equation}
								\label{eq110416_01}
q_1^* = \left(\frac{\gamma}{2(\gamma-1)}p_0\right)^*
= \frac{\gamma d p_0}{2\gamma d - 2 d - \gamma p_0} > p_0
\quad
\text{when}
\quad
\frac{\gamma}{2(\gamma-1)}p_0 < d.	
\end{equation}
Moreover, $2 < \sigma \le \hat{p}$. Hence we have
\begin{equation}
								\label{eq110416_02}
\sigma, q_1^* \in [\hat{p}/(\hat{p}-1), \hat{p}].	
\end{equation}

Set $p_1 = \min \{ \sigma, q_1^* \}$.
Then
$$
p_1 =
\left\{
\begin{aligned}
&\min\{\sigma, \tau^*\} \quad &\text{if}\quad\frac{\gamma}{2(\gamma-1)}p_0 \ge d,\\
&\min\{\sigma, \left(\frac{\gamma}{2(\gamma-1)}p_0\right)^*, \tau^*\}
\quad &\text{if}\quad\frac{\gamma}{2(\gamma-1)}p_0 < d.
\end{aligned}
\right.
$$
Observe that $u$ satisfies
$$
\left\{
  \begin{aligned}
    - D_i(a_{ij} D_j u) + \lambda u &= D_i h_i + h + \lambda u
 	\quad \text{in} \quad \Omega, \\
    \left(a_{ij}D_j v+h_i\right)\nu_i &= 0 \quad \text{on} \quad \partial \Omega,
  \end{aligned}
\right.
$$
where $a_{ij}$, $h_i$, and $h$ are those in \eqref{Linearized}.
Also observe that $u \in L_{q_1}(\Omega)$ because $\frac{\gamma}{2(\gamma-1)} < 1$.
Thus by Theorem \ref{liRes} along with \eqref{eq0224} and \eqref{eq0223} applied to \eqref{Linearized} we have $u \in W^1_{p_1}(\Omega)$ and
$$
\|u\|_{W^1_{p_1}(\Omega)}\le N\left(\|h_i\|_{L_\sigma(\Omega)}+\|h\|_{L_{q_1}(\Omega)} + \|u\|_{L_{q_1}(\Omega)}\right),	
$$
where $N=N(d,\mu,\hat{p}, \sigma, q_1, q_1^*, R_0, R_2, \lambda, \text{\rm diam}\Omega)$.
Bearing in mind the definitions of $h_i$ and $h$ as well as
using Theorem \ref{thm1}, we obtain \eqref{eq110405_01} unless
\begin{equation}
                                        \label{eq23.23}
\frac{\gamma}{2(\gamma-1)}p_0<d\quad
\text{and}
\quad
\left(\frac{\gamma}{2(\gamma-1)}p_0\right)^*< \min\{\sigma,\tau^*\}.
\end{equation}
In this case, $p_1 = q_1^*=\left(\frac{\gamma}{2(\gamma-1)}p_0\right)^*$ and, as seen in \eqref{eq110416_01}, $p_1 >p_0$.
Now, since $u \in W_{p_1}^1(\Omega)$, by \eqref{eq0211} it follows that
$$
h \in L_{q_2}(\Omega),
\quad
q_2 = \min\left\{ \frac{\gamma}{2(\gamma-1)}p_1, \tau \right\}.
$$
Note that $q_2>q_1$. We define $p_2=\min\{\sigma,q_2^*\}>p_1$. 
Then we see that \eqref{eq110416_02} is satisfied with $q_2^*$ in place of $q_1^*$
and $u \in L_{q_2}(\Omega)$.
By repeating the above argument, we obtain \eqref{eq110405_01} unless \eqref{eq23.23} holds with $p_1$ in place of $p_0$.
We continue, if necessary, repeating the above argument to obtain $p_3,p_4, \cdots$ with the recursion formula
$$
p_{k+1} = \left(\frac{\gamma}{2(\gamma-1)}p_k\right)^*
= \frac{\gamma d p_k}{2\gamma d - 2 d - \gamma p_k},
\quad
k=0,1,2,\cdots.
$$
Since
$$
p_{k+1} - p_k \ge \frac{p_0(\gamma p_0 + 2d-\gamma d)}{2\gamma d - 2d},
$$
there has to be an integer $k_0$ such that $p = p_{k_0} = \min\{\sigma, \tau^*\}$.
Note that \eqref{eq110416_02} holds true with $q_k^*$ in place of $q_1^*$ for all $k=1,\cdots,k_0$.
This allows us to use Theorem \ref{liRes} in the above iteration process with the same $\beta$ and $\rho$ in \eqref{eq0223} for all $k=1,\cdots,k_0$.
\end{proof}

\section{Functions in the class $H$}
							\label{HolderReg}

Throughout the section, the domain $\Omega$ is either $B_{r_0}$ or $B_{r_0}^+=\{x \in B_{r_0}, x_1 > 0\}$,
and $\Omega_r = \Omega \cap B_r$, where $B_r$ is concentric with $B_{r_0}$.
The results in this section are those in \cite[Chater 2, section 6]{LU}, where the interior  H\"{o}lder regularity is proved.
We slightly modified the statements in \cite{LU} so that they also work for the boundary H\"{o}lder regularity.
We also give precise parameters on which the constants in the statements depend. 
We omit here the proofs since they can be done in the same way as in \cite{LU}.

\begin{definition}
								\label{def110401}
Let $r_0 > 0$, $M>0$, $C > 0$, $\kappa > 0$, $q > d/2$ be real numbers,
and $\Omega = B_{\rho_0}$ or $\Omega = B_{\rho_0}^+$.
We say $v \in H(\Omega, M, C, \kappa, q)$ if $v \in W_2^1(\Omega)$ satisfies
$|v| \le M$ as well as the following two inequalities for any $r \in (0,r_0]$ and $\delta \in (0,1)$:
$$
\int_{\{v > k\} \cap \Omega_{r(1-\delta)}} |\nabla v|^2 \, dx
\le C \left(\frac{1}{\delta^2 r^{2-d/q}} \max_{\{v > k\} \cap \Omega_r} (v-k)^2 + 1\right)
|\{v > k\} \cap \Omega_r|^{1-1/q}
$$
for $k \ge \max_{\Omega_r}v - \kappa$, and
\begin{equation}
								\label{eq110401_01}
\int_{\{v < k\} \cap \Omega_{r(1-\delta)}} |\nabla v|^2 \, dx
\le C \left(\frac{1}{\delta^2 r^{2-d/q}} \max_{\{v < k\} \cap \Omega_r} (k-v)^2 + 1\right)
|\{v < k\} \cap \Omega_r|^{1-1/q}
\end{equation}
for $k \le \min_{\Omega_r}v + \kappa$.
\end{definition}

The following lemma is Lemma 2.3.5 in \cite{LU} with $\Omega = B_r$.
As noted there, it also works for any convex domains.

\begin{lemma}
Let $\Omega = B_r$ or $\Omega = B_r^+$.
Then for an arbitrary function $v$ in $W_1^1(\Omega)$
and for arbitrary $k$ and $l$ such that $k \le l$,
$$
(l-k) |\{ v > l \} \cap \Omega |^{1-1/d}
\le N \frac{r^d}{|\{v \le k\} \cap \Omega|}
\int_{\{ k < v \le l\} \cap \Omega}|\nabla v| \, dx,
$$
where $N = N(d)$.
\end{lemma}

\begin{lemma}
Let $v \in H(\Omega, M, C, \kappa, q)$.
Then there exists a $\theta_1 = \theta_1(d,C,q) > 0$ such that,
for any $\Omega_r \subset \Omega$
and for any number $k \ge \max_{\Omega_r}u(x) - \kappa$,
the inequality
$$
| \Omega_r \cap \{ v>k \}| \le \theta_1 r^d	
$$
implies
$$
| \Omega_{r/2} \cap \{ v>k+K/2 \}|=0,
$$
provided that
$$
K = \max_{\Omega_r}u - k \ge r^{1-\frac{d}{2q}}.
$$	
\end{lemma}

\begin{lemma}
								\label{lem110413}
Let $v \in H(\Omega, M, C, \kappa, q)$.
Then there exists a positive integer $s=s(d,C,q)$ such that, for any $\Omega_r \subset \Omega_{4r} \subset \Omega$,
at least one of the following two inequalities holds:
$$
\text{\rm osc}_{\Omega_r} v \le 2^s r^{1-\frac{d}{2q}},
$$
$$
\text{\rm osc}_{\Omega_r} v \le \left(1- \frac{1}{2^{s-1}}\right) \text{\rm osc}_{\Omega_{4r}}v.
$$
\end{lemma}

\begin{theorem}
								\label{thm110413}
Let $v \in H(\Omega, M, C, \kappa, q)$. 
Then for all $r \le r_0$, we have
$$
\text{\rm osc}_{\Omega_r} v
\le N \left(\frac{r}{r_0}\right)^\alpha,
$$
where 
$$
\alpha = \min\left\{-\log_4 \left(1-\frac{1}{2^{s-1}}\right), 1-\frac{d}{2q}\right\},
$$
$$
N=4^\alpha \max \left\{ \text{\rm osc}_{\Omega_{r_0}} u, 2^s r_0^{1-\frac{d}{2q}}\right\},
$$
and the number $s$ is taken from Lemma \ref{lem110413}.
\end{theorem}

\section{Reverse H\"{o}lder's inequality}
							\label{elliptic-sec}
							
Recall the definition of $\gamma$ in \eqref{eq110411_04}.
Also recall that, throughout the section, $\Omega$ is a bounded domain satisfying Assumption \ref{assump2}.
Let
$$
c = (u)_{B_R(x_0)} = \dashint_{B_R} u \, dx
\quad
\text{and}
\quad
q= \frac{2d}{d+2}.
$$
Then by the Poincar\'e inequality, we have
$$
\int_{B_R} |u-c|^2 \, dx \le N(d) R^{d+2-2d/q}
\left(\int_{B_R} |\nabla u|^{q} \, dx\right)^{2/q},
$$
$$
\int_{B_R} |u-c|^\gamma \, dx \le N(d,\gamma) R^{d+\gamma-\gamma d/2} \left(\int_{B_R} |\nabla u|^2 \, dx\right)^{\gamma/2}.
$$
These inequalities also hold true if $B_R$ is replaced by $B_R^+=\{|x|<R: x_1>0\}$.
As before, we write $\Omega_r(x) = \Omega \cap B_r(x)$.

\begin{lemma}
								\label{lem110417}
Let $\Omega$ be a Lipschitz domain satisfying Assumption \ref{assump2},
$u \in W_2^1(\Omega)$,
and $x_0 \in \partial \Omega$.
Then for $R \le R_0$, we have 
\begin{equation}
								\label{eq110415_01}
\int_{\Omega_r(x_0)} |u-c|^2 \, dx \le N_1 R^{d+2-2d/q}
\left(\int_{\Omega_R(x_0)} |\nabla u|^q \, dx\right)^{2/q},	
\end{equation}
$$
\int_{\Omega_r(x_0)} |u-c|^\gamma \, dx \le N_2 R^{d+\gamma-\gamma d/2} \left(\int_{\Omega_R(x_0)} |\nabla u|^2 \, dx\right)^{\gamma/2},
$$
where $r = \frac{R}{2(1+\beta^2)}$, $q= 2d/(d+2)$, $N_1=N_1(d,\beta)$, 
$N_2 = N_2(d,\gamma,\beta)$, and
$$
c = (u)_{\Omega_r(x_0)}
= \dashint_{\Omega_r(x_0)} u \, dx.
$$
\end{lemma}

\begin{proof}
Without loss of generality we assume that $x_0 = 0$
and $\varphi(0)=0$, where $\varphi$ is a Lipschitz function 
such that $\Omega_{R_0} = \Omega \cap B_{R_0}
= \{x \in B_{R_0}: x_1 > \varphi(x')\}$.
Let
$\Phi(y) = (y_1+\varphi(y'),y')$
and
$\Phi^{-1}(x) = \Psi(x) = (x_1-\varphi(x'), x')$.
Also let $v(y) = u(\Phi(y))$.
Then by \eqref{eq110413_03} we can say $v \in W_2^1(B_{R_1})$,
where $R_1=\frac{R}{\sqrt{2(1+\beta^2)}}$.
From the Poincar\'e inequality above for a half ball it follows that
$$
\int_{B_{R_1}^+} |v-(v)_{B_{R_1}^+}|^2 \, dy
\le N R^{d+2-2d/q}\left( \int_{B_{R_1}^+} |\nabla v|^q \, dy \right)^{2/q}.
$$
Here
$$
(v)_{B_{R_1}^+} = \dashint_{B_{R_1}^+} v(y) \, dy.
$$
From this and the set inclusions in \eqref{eq110413_03} we see that
$$
\int_{\Omega_r} |u-(v)_{B_{R_1}^+}|^2 \, dx
\le \int_{ \Phi(B_{R_1}) \cap \{x_1 > \varphi(x')\} }
|u-(v)_{B_{R_1}^+}|^2 \, dx
= \int_{B_{R_1}^+}
|v-(v)_{B_{R_1}^+}|^2 \, dy
$$
$$
\le N R^{d+2-2d/q}\left( \int_{B_{R_1}^+} |\nabla v|^q \, dy \right)^{2/q}
\le N R^{d+2-2d/q}\left( \int_{\Omega_R} |\nabla u|^q \, dy \right)^{2/q},
$$
where $N=N(d,\beta)$.
Now the inequality \eqref{eq110415_01} follows because 
$$
\int_{\Omega_r}|u-(u)_{\Omega_r}|^2 \, dx
\le \int_{\Omega_r}|u - C|^2 \,dx
$$
for any constant $C$.
The other inequality follows similarly. The lemma is proved.
\end{proof}

Theorem \ref{thm1} is proved by the following proposition combined with Proposition V.1.1 in \cite{Gi83}. 
Also see the proof of Theorem 3.6 in \cite{DK10}.

\begin{proposition}
								\label{prop2}
Let $R \le R_0$, $u\in W^1_2(\Omega)$ be a weak solution to \eqref{eq3.24},
$f\in L_2(\Omega)$, and $g\in L_{\frac{\gamma}{\gamma-1}}(\Omega)$.
Then, for any $\Omega_R(x_0)$,
where either $B_R(x_0) \subset \Omega$
or $x_0 \in \partial\Omega$,
we have
\begin{multline*}
\dashint_{\Omega_{\varrho R}(x_0)} (|\nabla u|^2+|u|^\gamma)
\le N \left( \dashint_{\Omega_R(x_0)} |\nabla u|^q + |u|^{\gamma q/2} \right)^{\frac{2}{q}}
+ N \dashint_{\Omega_R(x_0)} \left(|f|^2 + |F|^2 \right)\\
+ N R^{d\gamma\left(\frac{1}{\gamma}-\frac12\right) + \gamma}
\left(\int_{\Omega_R(x_0)} |\nabla u|^2\right)^{\frac{\gamma}{2}-1}
\left(\dashint_{\Omega_R(x_0)} |\nabla u|^2\right),
\end{multline*}
where $\varrho = \frac{1}{4(1+\beta^2)} \in (0,1)$, 
$q = \frac{2d}{d+2}$,
$F = |g|^{\frac12\frac{\gamma}{\gamma-1}}$, and
$N = N(d,\mu,\mu_1,\mu_2,\gamma,\beta)$.
\end{proposition}

Note that $R^{d\gamma\left(\frac{1}{\gamma}-\frac12\right) + \gamma} = 1$ if $d > 2$
and $R^{d\gamma\left(\frac{1}{\gamma}-\frac12\right) + \gamma} = R^2$ if $d=2$.

\begin{proof}
We only show the case $x_0 \in \partial\Omega$.
The other case follows the same lines.	
Let $\eta_0\in C_0^\infty(\bR^d)$ be a function satisfying $0 \le \eta_0 \le 1$ and 
$$
\eta_0 =
\left\{
\begin{aligned}
1 \quad \text{for} \quad |x| \le \frac{1}{4(1+\beta^2)},\\
0 \quad \text{for} \quad |x| \ge \frac{1}{2(1+\beta^2)}.
\end{aligned} 
\right.
$$
Set  $r = \frac{R}{2(1+\beta^2)}$,
$$
c = (u)_{\Omega_r(x_0)} = \dashint_{\Omega_r(x_0)} u \, dx,
\quad
\text{and}
\quad
\eta=\eta_0(R^{-1}(\cdot-x_0)).
$$
Using a test function $(u-c)\eta^2$,
we have
$$	
\int_{\Omega_R(x_0)} A_{ij}(x,u) D_i\left[(u-c) \eta^2\right] D_j u
$$
$$
+ \int_{\Omega_R(x_0)} a_i(x,u) D_i\left[(u-c) \eta^2\right]
= \int_{\Omega_R(x_0)} b(x,u,\nabla) (u-c) \eta^2.
$$
That is,
\begin{multline}
								\label{eq110418_01}
\int_{\Omega_R(x_0)} A_{ij} \eta (D_i u) \eta (D_j u)
= -\int_{\Omega_R(x_0)} 2 A_{ij} (u-c) \eta (D_i \eta) (D_j u)
\\
-\int_{\Omega_R(x_0)} a_i D_i\left[(u-c) \eta^2\right]
+ \int_{\Omega_R(x_0)} b (u-c) \eta^2
=: J_1 + J_2 + J_3.	
\end{multline}
We estimate $J_1$, $J_2$ and $J_3$ by using Young's inequality
and the conditions on $A_{ij}$, $a_i$, and $b$.

Estimate of $J_1$:
$$
J_1 \le 2\mu^{-1} \int_{\Omega_R(x_0)} |\nabla u| |u-c| \eta |\nabla  \eta|
\le \frac \mu {16} \int_{\Omega_R(x_0)} \eta^2 |\nabla u|^2
+ N \int_{\Omega_R(x_0)} |u-c|^2 |\nabla \eta|^2.
$$

Estimate of $J_2$:
\begin{align*}
J_2 &\le \mu_1 \int_{\Omega_R(x_0)} |u|^{\gamma/2} |\nabla u| \eta^2
+ \mu_1 \int_{\Omega_R(x_0)} |f| |\nabla u| \eta^2\\
&\quad +  2\mu_1 \int_{\Omega_R(x_0)} |u|^{\gamma/2} |u-c| |\nabla \eta| \eta
+ 2\mu_1 \int_{\Omega_R(x_0)} |f| |u-c| |\nabla \eta| \eta\\
&\le \frac \mu {16} \int_{\Omega_R(x_0)}|\nabla u|^2  \eta^2 +
N \int_{\Omega_R(x_0)} |u|^{\gamma} \eta^2
+ N \int_{\Omega_R(x_0)} |f|^2 \eta^2
+ N \int_{\Omega_R(x_0)} |u-c|^2 |\nabla \eta|^2.
\end{align*}

Estimate of $J_3$:
\begin{align*}
J_3 &\le \mu_2 \int_{\Omega_R(x_0)} |\nabla u|^{2(1-1/\gamma)} |u-c| \eta^2
+ \mu_2 \int_{\Omega_R(x_0)} |u|^{\gamma-1} |u-c| \eta^2
+ \mu_2 \int_{\Omega_R(x_0)}|g| |u-c| \eta^2\\
&\le \frac \mu {16} \int_{\Omega_R(x_0)} |\nabla u|^{2} \eta^2
+ N \int_{\Omega_R(x_0)} |u-c|^{\gamma} \eta^2
+ N \int_{\Omega_R(x_0)} |u|^{\gamma} \eta^2
+ N \int_{\Omega_R(x_0)}|g|^{\frac \gamma {\gamma-1}} \eta^2.
\end{align*}
From these estimates of $J_i$, $i=1,2,3$, and the inequality \eqref{eq110418_01} along with the ellipticity condition in \eqref{eq110409_02} we have
$$
\int_{\Omega_R(x_0)} |\nabla u|^2 \eta^2
\le N \int_{\Omega_R(x_0)} |u-c|^2 |\nabla \eta|^2
+ N \int_{\Omega_R(x_0)} |u-c|^\gamma \eta^2
$$
$$
+ N \int_{\Omega_R(x_0)} |u|^\gamma \eta^2
+ N \int_{\Omega_R(x_0)} \left(|f|^2+|g|^{\frac \gamma {\gamma-1}}\right) \eta^2
:= N(I_1 + I_2 + I_3 + I_4),
$$
where $N=N(\mu,\mu_1,\mu_2)$. 
Now we get estimates for $I_1$, $I_2$, and $I_3$ as follows.

Estimate of $I_1$: Recall the definition of $\eta$, $r = \frac{R}{2(1+\beta^2)}$,
and $q = \frac{2d}{d+2}$.
Then by Lemma \ref{lem110417},
$$
I_1 \le N R^{-2}\int_{\Omega_r(x_0)} |u-c|^2 \le
N R^d
\left(\dashint_{\Omega_R(x_0)} |\nabla u|^q\right)^{2/q}.
$$

Estimate of $I_2$: Again by Lemma \ref{lem110417},
$$
I_2 \le \int_{\Omega_r(x_0)} |u-c|^\gamma
\le N R^{d+\gamma - \gamma d/2}\left(\int_{\Omega_R(x_0)} |\nabla u|^2 \right)^{\gamma/2}.
$$

Estimate of $I_3$: First note that
$$
\int_{\Omega_r(x_0)} |c|^{\gamma} \, dx
\le N R^d \left(\dashint_{\Omega_r(x_0)} |u| \, dx\right)^\gamma
\le N R^d \left(\dashint_{\Omega_R(x_0)} |u|^{\gamma q/2} \, dx\right)^{2/q},
$$
where, in the last inequality,
we have used the fact that
$|\Omega_r(x_0)| \ge |B_{r_1}^+|$, $r_1 = R \left(2(1+\beta^2)\right)^{-3/2}$. 
Hence
\begin{multline}
                                    \label{eq13.33}	
I_3 \le N \int_{\Omega_r(x_0)} |u-c|^\gamma
+ N \int_{\Omega_r(x_0)} |c|^{\gamma}
\\
\le N R^{d+\gamma - \gamma d/2}\left(\int_{\Omega_R(x_0)} |\nabla u|^2 \right)^{\gamma/2} 
+ R^d \left(\dashint_{\Omega_R(x_0)} |u|^{\gamma q/2} \, dx\right)^{2/q}.
\end{multline}
Therefore,
$$
\int_{\Omega_R(x_0)} |\nabla u|^2 \eta^2
\le N R^d \left( \dashint_{\Omega_R(x_0)} |\nabla u|^q \right)^{\frac{2}{q}}
$$
$$
+ N R^d \left( \dashint_{\Omega_R(x_0)} |u|^{\gamma q/2} \right)^{\frac{2}{q}}
+ N \int_{\Omega_R(x_0)} \left(|f|^2 + |F|^2 \right)
$$
$$
\quad + N R^{d\gamma\left(\frac{1}{\gamma}-\frac12\right) + \gamma}
\left(\int_{\Omega_R(x_0)} |\nabla u|^2\right)^{\frac{\gamma}{2}-1}
\left(\int_{\Omega_R(x_0)} |\nabla u|^2\right),
$$
where $N=N(d,\mu,\mu_1,\mu_2,\gamma,\beta)$ and
$$
F =
|g|^{\frac12\frac{\gamma}{\gamma-1}}.
$$
Finally, we obtain the desired inequality in the proposition
by adding the $I_3$ term to the above inequality, using \eqref{eq13.33}, and diving all terms by $R^d$.
\end{proof}

\section*{Acknowledgement}
I would like to thank Hongjie Dong for initiating the writing of this paper,
and Dian K. Palagachev for informing me of his recent paper \cite{PalSof11}.

\bibliographystyle{plain}

\def\cprime{$'$}\def\cprime{$'$} \def\cprime{$'$} \def\cprime{$'$}
  \def\cprime{$'$} \def\cprime{$'$}

\end{document}